\documentclass[a4paper, 11pt]{article}

\usepackage{amsmath, amscd, amsfonts, amssymb, amsthm, latexsym, url}
\usepackage{graphicx}
\usepackage{hyperref}
\usepackage{setspace}
\input{xy}
\xyoption{all}
\usepackage{enumerate}

\makeatletter
\def\url@leostyle{%
  \@ifundefined{selectfont}{\def\UrlFont{\sf}}{\def\UrlFont{\small\ttfamily}}}
\makeatother
\usepackage[left=3.5cm, right=3.5cm, top=3.5cm, bottom=3.5cm]{geometry}

\usepackage{sectsty}
\sectionfont{\large\bf}
\chapterfont{\large\bf\centering}
\chaptertitlefont{\LARGE\centering}
\partfont{\centering\bf}
\paragraphfont{\bf}

\theoremstyle{plain}
\newtheorem{thm}{Theorem}[section]
\newtheorem{lem}[thm]{Lemma}
\newtheorem{corollary}[thm]{Corollary}
\newtheorem{prop}[thm]{Proposition}

\theoremstyle{definition}
\newtheorem{df}[thm]{Definition}
\newtheorem{rmk}[thm]{Remark}

\newcommand{\roundb}[1]{\! \left(\! \left( #1 \right)\!\right)}

\newcommand{\curlyb}[1]{\! \left\{\! \left\{ #1 \right\}\!\right\}}

\newcommand{\OO}{\mathcal{O}}

\newcommand{\pp}{\mathfrak{p}}

\newcommand{\QQ}{\mathbb{Q}}
\newcommand{\ZZ}{\mathbb{Z}}

\newcommand{\car}{\mathrm{char}\:}

\newcommand{\Tr}{\mathrm{Tr}\:}

\newcommand{\into}{\hookrightarrow}

\newcommand{\arxiv}[1]{\href{http://arxiv.org/abs/#1}{{\tt \small arXiv:#1}}}

\title{Locally convex structures on higher local fields} 
\date{\today}
\author{Alberto C\'amara\footnote{The author is supported by a Doctoral Training Grant at the University of Nottingham.}}

\begin{document}

\maketitle

\begin{abstract}
  We establish how a higher local field can be described as a locally
  convex vector space once an embedding of a local field into it has been
  fixed. This extends previous results that had been obtained in the
  two-dimensional case. In particular, we study bounded
  and compactoid submodules of these fields and establish a self-duality
  result once a suitable topology on the dual space has been introduced.
\end{abstract}


\section*{Introduction}

In \cite{camara-fa2dlfs} we explained how characteristic zero two-dimensional local
fields may be regarded as locally convex vector spaces once an embedding
of a local field into them has been fixed.

This note, designed as a natural continuation of that work, explains how
the locally convex approach to higher topologies works for a higher local
field of arbitrary dimension.

It is perhaps necessary to explain the need to treat the arbitrary
dimensional case separately. The two-dimensional case often supplies the
first step of induction and therefore it is a good idea to treat it
first. At the same time, the cases for dimension
greater than two quickly turn into a rather involved exercise in notation
and the application of arguments which are familiar from the case $n=2$.
As such, the proof of many results in this note often refers to
\cite{camara-fa2dlfs} for the case $n=2$ and then indicates how to
proceed by induction.

Furthermore, there are many relevant functional analytic properties which
may be shown to hold in the two-dimensional case and which fail in
greater dimension or which one could only expect to hold in few
particular cases; being bornological, reflexive or nuclear is an example
of such properties.

However, it is possible to give explicit bases for the bornologies of
bounded and compactoid $\OO$-submodules of $F$. We also show an
explicit self-duality result in Theorem \ref{thm:selfduality}
which generalises \cite[Theorem 6.2]{camara-fa2dlfs}.

We not only often refer to concepts and results explained in
\cite{camara-fa2dlfs}, but also assume the reader to be
familiar with it. In particular, we reviewed the definitions and results of the theory of
locally convex vector spaces over a local field which are needed for this
work in
\cite[\textsection 1]{camara-fa2dlfs}. Besides that,
\cite{perez-garcia-schikof-locally-convex-spaces-nonarchimedean-valued-fields}
and \cite{schneider-non-archimedean-functional-analysis} contain suitable introductions to the topic.

We outline the contents of this work. \textsection
\ref{sec:categoryofhdlfs} is shaped very much after \cite[\textsection
2]{camara-fa2dlfs} and summarizes certain results from the
structure theory of higher local fields. At the end of this section we
focus our attention on a {\it standard} higher local field, which is one
of the form
\begin{equation*}
  F =
  K\curlyb{t_1}\cdots\curlyb{t_r}\roundb{t_{r+1}}\cdots\roundb{t_{n-1}}
\end{equation*}
and deduce our results first in this case.

Section \textsection \ref{sec:highertops} explains how the higher
topology on $F$ is locally convex. \textsection
\ref{sec:firstproperties} exposes facts which are either immediate 
consequences of the results in the previous section or facts which were
already known about higher topologies.

Sections \textsection \ref{sec:boundedsubmodules} and \textsection
\ref{sec:compactoidsubmodules} deal with the study of bounded
submodules and compactoid submodules of $F$,
respectively.

We study duality issues in \textsection \ref{sec:duality}. In Theorem
\ref{thm:selfduality} we prove 
self-duality after topologizing the dual space adequately, generalising
the work that was done for the two-dimensional case in \cite[Theorem
6.2]{camara-fa2dlfs}. We also describe polars and pseudo-polars 
of relevant submodules of $F$.

We explain how the results obtained in the previous sections may be
extended from $F$ to an arbitrary $n$-dimensional local field in
\textsection \ref{sec:generalcase} and we dedicate a few words to the positive
characteristic and archimedean cases in section \textsection
\ref{sec:positivecharcase}.

Finally, we discuss some interesting questions and directions of work specifically
related to this note in \textsection \ref{sec:future}.

\paragraph{Notation.} Whenever $F$ is a complete discrete valuation
field, we will denote by $\OO_F, \pp_F, \pi_F, \overline{F}$ its ring of
integers, the unique nonzero prime ideal in the ring of integers, an
element of valuation one and the residue field, respectively. 

Throughout the text, $K$ will denote a characteristic zero local
field, that is, a finite extension of $\QQ_p$ for some prime $p$. The cardinality of the finite field $\overline{K}$ will be denoted by $q$. The absolute value of $K$
will be denoted by $|\cdot|$, normalised so that $|\pi_K| = q^{-1}$. Due
to far too frequent apparitions in the text, we will ease notation by
letting $\OO := \OO_K$ and $\pp := \pp_K$.

The conventions $\pp^{-\infty} = K$, 
$\pp^\infty = 0$ and $q^{-\infty} = 0$ will be used.

The main object of study of this work is a field inclusion $K \subset F$
where $F$ is an $n$-dimensional local field. See \textsection
\ref{sec:categoryofhdlfs} for details.

\paragraph{Acknowledgements.} I thank Thomas Oliver for carefully reading
this note and pointing out several improvements. I also thank my
supervisor Ivan Fesenko for his guidance and encouragement. Finally, I am
also grateful to Cristina P\'erez Garc\'ia for several interesting
conversations regarding the theory of locally convex nonarchimedean
spaces.

\section{Higher local fields arising from arithmetic contexts}
\label{sec:categoryofhdlfs}

A {\it zero-dimensional local field} is a finite field. An {\it $n$-dimensional
local field}, for $n \geq 1$, is a complete discrete valuation field $F$
such that $\overline{F}$ is an $(n-1)$-dimensional local field. Thus, a
local field in the usual sense is a one-dimensional local field.

An $n$-dimensional local field $F$ determines then a collection of fields
$F_i$, $i \in \left\{ 0, \ldots, n \right\}$, by letting $F_n = F$,
$\overline{F_i} = F_{i-1}$ for $1 \leq i \leq n$; being $n$-dimensional
is then determined by the finiteness of $F_0$

An excellent introduction to this topic may be found in
\cite{morrow-intro-hlf}, we also often refer to results explained in
\cite{ihlf}.

In \cite[\textsection 2]{camara-fa2dlfs}, we explained how, as
was first introduced in \cite{morrow-explicit-approach-to-residues},
it is a good idea to regard two-dimensional local fields arising from an
arithmetic context as vector spaces over a local field.

The construction may be generalised to an arithmetic scheme of any
dimension as follows. Let $S$ be the spectrum of the ring of integers of a
number field and $f: X \to S$ be an arithmetic scheme of dimension $n$
(for our purposes, it is enough to suppose that $X$ is an $n$-dimensional
regular scheme and that $f$ is projective and
flat). Given a complete flag of irreducible subschemes $\eta_n \in
\overline{\left\{ \eta_{n-1} \right\}} \subset \cdots \subset
\overline{\left\{ \eta_0 \right\}} = X$, and assuming for simplicity that 
$\eta_n$ is regular in $\overline{\left\{ \eta_i \right\}}$ for each $0
\leq i \leq n-1$, define $A^n
= \widehat{\OO_{X,\eta_n}}$ and
\begin{equation*}
  A^i = \widehat{A_{\eta_i}^{i+1}}, \quad i \in \left\{ 0, \ldots, n-1 \right\}. 
\end{equation*}
It can be shown \cite[Remark 6.12]{morrow-intro-hlf} that $F = A^0$ is an
$n$-dimensional local field. The ring homomorphism $\OO_{S, f(x)} \to
\OO_{X,x}$ induces a field embedding $K \into F$, where $K =
\mathrm{Frac}\left( \widehat{\OO_{S,f(x)}} \right)$. Our conclusion is
that whenever $n$-dimensional local fields arise from an
arithmetic-geometric setting they always come equipped with a prefixed
embedding of a local field into them.

This justifies our decision to fix a characteristic zero local field $K$
and to study $n$-dimensional local fields not as fields $F$, but as pairs
of a field $F$ and a field embedding $K \into F$. We shall refer to such
a pair as an {\it $n$-dimensional local field over $K$}. A morphism of higher
local fields over $K$ is therefore a commutative diagram of field embeddings
\begin{equation*}
  \xymatrix{
  F_1 \ar@{->}[r] & F_2 \\
  K  \ar@{->}[u] \ar@{->}[ur] & 
  }
\end{equation*}
where $F_1$ and $F_2$ are higher local fields and $F_1 \to F_2$ is an extension of complete discrete valuation rings.

The structure of higher local fields is explained in
\cite{madunts-zhukov-topology-hlfs}, \cite[\textsection 1]{ihlf} or
\cite[Theorem 2.18]{morrow-intro-hlf}. They may be classified using
Cohen structure theory for complete rings.
In particular, we have the following possibilities: 
\begin{enumerate}
  \item If $\car F$ is positive, then it is possible to choose $t_1,
    \ldots, t_n \in F$ such that $F \cong
    F_0\roundb{t_1}\cdots\roundb{t_n}$. In this work we assume $\car F
    =0$, and only treat the positive characteristic case in \textsection
    \ref{sec:positivecharcase}. 
  \item If $\car F_1 = 0$, then there are $t_1, \ldots, t_{n-1} \in
    F$ such that 
    \begin{equation*}
      F \cong F_1\roundb{t_1}\cdots\roundb{t_{n-1}}.
    \end{equation*}
    Moreover, if an embedding of fields $K \into F$ has been fixed, then
    we have a finite extension $K \into F_1$.
  \item If none of the above holds, then there is a unique $r \in \left\{
    1, \ldots, n-1 \right\}$ such that $\car F_{r+1} \neq \car
    F_{r}$. Then there is a characteristic zero local field $L$ and
    elements $t_1, \ldots, t_{n-1} \in F$ such that $F$ is a finite
    extension of 
    \begin{equation*}
      L\curlyb{t_1}\cdots\curlyb{t_r}\roundb{t_{r+1}}\cdots\roundb{t_{n-1}}.
    \end{equation*}
    Moreover, if $\car F_0 = p$, $L$ may be chosen to be the unique 
    unramified extension of $\mathbb{Q}_p$ with residue field $F_0$. In
    this work we will not require this fact, but simply use the fact that
    given an embedding $K \into F$, there is a finite subextension
    \begin{equation*}
      K\curlyb{t_1}\cdots\curlyb{t_r}\roundb{t_{r+1}}\cdots\roundb{t_{n-1}}
      \into F.
    \end{equation*}
\end{enumerate}

\paragraph{Notation.} We fix from now on, and until the beginning of
\textsection \ref{sec:generalcase},
\begin{equation*}
  F = K\curlyb{t_1}\cdots\curlyb{t_r}\roundb{t_{r+1}}\cdots\roundb{t_{n-1}}
\end{equation*}
with $0 \leq r \leq n-1$. The extremal case $r = 0$ (resp. $r=n-1$)
stands for $F = K\roundb{t_1}\cdots\roundb{t_{n-1}}$ (resp. $F =
K\curlyb{t_1}\cdots\curlyb{t_{n-1}}$). We also let
\begin{equation*}
  L=K\curlyb{t_1}\cdots\curlyb{t_r}\roundb{t_{r+1}}\cdots\roundb{t_{n-2}},
\end{equation*}
by which we simply mean that $L$ is the subfield of $F$ consisting of
power series in $t_1, \ldots, t_{n-2}$.

It will be extremely convenient to use multi-index notation. For this
purpose, let $I = \ZZ^{n-1}$ and $J = \ZZ^{n-2}$. For $l\in \left\{
1,\ldots,n-1\right\}$, if we fix indices
$(i_l,\ldots,i_{n-1}) \in \ZZ^{n-l-1}$, we will denote 
\begin{equation*}
  I(i_l, \ldots, i_{n-1}) = \left\{ \alpha \in I;\; \alpha = \left( j_1,
  \ldots, j_{l-1}, i_l,
  \ldots, i_{n-1} \right) \text{ for some } (j_1, \ldots,j_{l-1}) \in
  \ZZ^{l-1} \right\}.
\end{equation*}

Any element $x \in F$ can be written uniquely as a power series
\begin{equation*}
  x = \sum_{i_{n-1} \gg -\infty} \cdots \sum_{i_{r+1} \gg -\infty }
  \sum_{i_r \in \ZZ} \cdots \sum_{i_1 \in \ZZ} x_{i_1, \ldots,
  i_{n-1}} t_1^{i_1} \cdots t_{n-1}^{i_{n-1}},
\end{equation*}
with $x_{i_1, \ldots, i_{n-1}} \in K$. We will abbreviate such an expression to
\begin{equation*}
 x = \sum_\alpha x(\alpha) t^\alpha,
\end{equation*}
for $\alpha \in I$ and $x(\alpha) \in K$. Finally, for $\alpha =
(i_1, \ldots, i_{n-1}) \in I$, denote $-\alpha = (-i_1,\ldots,
-i_{n-1})$.

Several proofs will use induction arguments. For such, it will be
convenient to denote elements of $L$ as $\sum_{\beta} x(\beta)
t^\beta$ for $\beta \in J$ and $x(\beta) \in K$, with this notation being
analogous to that adopted for elements in $F$. The statement
$\alpha=(\beta,i)$ for $\alpha \in I$, $\beta \in J$ and $i \in \ZZ$
means that if $\beta=(i_1, \ldots, i_{n-2})$, then $\alpha = (i_1,
\ldots, i_{n-2}, i)$.

When necessary, $I$ will be ordered with the inverse lexicographical
order, that is:
$(i_1, \ldots, i_{n-1}) < (j_1, \ldots, j_{n-1})$ if and only if for an
index $l \in \left\{ 1, \ldots, n-1 \right\}$ we have $i_l < j_l$ and
$i_m = j_m$ for $l < m \leq n-1$.

By a {\it net}, we will refer to a set indexed by $I$ or $J$ (more
generally, a net in a topological space is a subset indexed by a
directed set, but we will only use this more general notion in Proposition
\ref{prop:boundedsubmodsarecomplete}). We will
construct objects such as $\OO$-submodules of and seminorms on $F$
attached to a given net in $\ZZ \cup \left\{ \pm \infty \right\}$.
Instead of using notation $(n_\alpha)_{\alpha \in
I}$, which is standard for sequences and was used thoroughly in
\cite{camara-fa2dlfs}, we will denote the elements of a net by
$(n(\alpha))_{\alpha \in I}$, the 
net of coefficients $(x(\alpha))_{\alpha \in I} \subset K$ attached to an element
$x = \sum_\alpha x(\alpha) t^\alpha \in F$ being a first example. We will
ease notation by letting, when given a net $(n(\alpha))_{\alpha \in
I}\subset \ZZ \cup \left\{ \pm \infty \right\}$, $n(\beta,i) :=
n\left( \left( \beta,i \right) \right)$ for $\beta \in J$, $i \in \ZZ$.

\section{Local convexity of higher topologies}
\label{sec:highertops}

The construction of the higher topology on $F$ is explained for example in
\cite[\textsection 4]{morrow-intro-hlf} and \cite[\textsection
1]{madunts-zhukov-topology-hlfs}. It revolves around two basic
constructions.

First suppose that a $L$ is a field on which a translation
invariant and Hausdorff topology has been defined. Let $\left\{ U_i
\right\}_{i \in \ZZ}$ be a sequence of neighbourhoods of zero of $L$,
with the property that there is an index $i_0 \in \ZZ$ such that
$U_i = L$ for all $i \geq i_0$. The sets of the form
\begin{equation}
  \label{eqn:basicnhoodequalchar}
  \sum_{i \in \ZZ} U_i t^i := \left\{ \sum_{i \gg -\infty}
  x_i t^i;\; x_i \in U_i \text{ for all } i\right\}
\end{equation}
describe a basis of neighbourhoods of zero for a translation invariant
Hausdorff topology on $L\roundb{t}$ \cite[\textsection 1]{madunts-zhukov-topology-hlfs}.

Second, suppose that $L$ is a complete discrete valuation field with
$\car L \neq \car \overline{L}$, so that $L\curlyb{t}$ is a complete
discrete valuation field of mixed characteristic.
Suppose that a
translation invariant and Hausdorff topology has been defined on $L$.
Let $\left\{ V_i \right\}_{i \in \ZZ}$ be a sequence of neighbourhoods of 
zero of $L$ satisfying the following two conditions:
\begin{enumerate}
  \item There is $c \in \ZZ$ such that $\pp_L^c \subset V_i$ for every $i
    \in \ZZ$.
  \item For every $l \in \ZZ$ there is $i_0 \in \ZZ$ such that for every
    $i \geq i_0$ we have $\pp_L^l \subset V_i$. This condition simply
    means that as $i \to \infty$ the neighbourhoods of zero $V_i$ become
    bigger and bigger. We will denote this condition by $V_i \to L$ as
    $i \to \infty$.
\end{enumerate}

The sets of the form
\begin{equation}
  \label{eqn:basicnhoodmixedchar}
  \sum V_i t^i := \left\{ \sum_{i \in \ZZ} x_i t^i \in L\curlyb{t};\; x_i
  \in V_i \text{ for all } i\right\}
\end{equation}
constitute the basis of neighbourhoods of zero for a translation
invariant and Hausdorff topology on $L\curlyb{t}$
\cite[\textsection 1]{madunts-zhukov-topology-hlfs}.

The procedure for topologizing $F$ is an inductive application of the two
constructions specified above. Namely, we endow $K$ with its $\pp$-adic
topology, and for every $k \in \left\{ 1,
\ldots, r \right\}$, we apply the second construction inductively on
$E\curlyb{t_k}$, with $E = K\curlyb{t_1}\cdots\curlyb{t_{k-1}}$.
For $k \in \left\{ r+1, \ldots, n-1 \right\}$, we apply inductively the
first construction on $E\roundb{t_k}$, with $E =
K\curlyb{t_1}\cdots\curlyb{t_r}\roundb{t_{r+1}}\cdots\roundb{t_{k-1}}$.

The resulting topology on $F$ is called the {\it higher topology}.

If $r < n$ the higher topology on $F$
depends on the choice of a coefficient field, that is, a field
inclusion $\overline{F} \subset F$ \cite[\textsection
1.4]{madunts-zhukov-topology-hlfs}. This is due to the fact that, since in this case $\car
\overline{F} = 0$ and $\overline{F}$ is transcendental over
$\mathbb{Q}$, there are infinitely many choices for such an embedding
\cite[II.5]{fesenko-vostokov-local-fields}.
In our description of the higher topology, we are implicitly choosing
an isomorphism $\overline{F} \cong K\curlyb{t_1}\cdots\curlyb{t_r}
\roundb{t_{r+1}}\cdots\roundb{t_{n-2}}$. In the two-dimensional equal
characteristic case, there is a unique coefficient field which factors
the field embedding $K \into F$, namely the algebraic closure of $K$ in
$F$ \cite[\textsection 7.1]{camara-fa2dlfs}. Such a choice is not possible whenever $n \geq 3$.

\begin{prop}
  \label{prop:highertopislocconv}
  The higher topology on $F$ is locally convex. It may be described as
  follows. For any net $( n(\alpha) )_{\alpha \in I} \subset
  \ZZ \cup \left\{ -\infty \right\}$ subjected to the conditions:
  \begin{enumerate}
    \item For any $l \in \left\{ r+1, \ldots, n-1 \right\}$ and fixed
      indices $i_{l+1}, \ldots, i_{n-1} \in \ZZ$, there is a $k_0 \in \ZZ$ such
      that for every $k \geq k_0$ we have
      \begin{equation*}
	n(\alpha) = -\infty \quad \text{for all } \alpha \in I(k,
	i_{l+1}, \ldots, i_{n-1} ).
      \end{equation*}
    \item For any $l \in \left\{ 1, \ldots, r \right\}$ and fixed indices
      $i_{l+1}, \ldots, i_{n-1}$, there is an integer $c \in \ZZ$ such
      that
      \begin{equation*}
	n(\alpha) \leq c\quad \text{for every } \alpha \in I(i_{l+1}, \ldots, i_{n-1}),
      \end{equation*}
      and we have that
      \begin{equation*}
	n(\alpha) \to -\infty, \quad \alpha \in I(k, i_{l+1}, \ldots,
	i_{n-1}), \text{ as } k \to \infty.
      \end{equation*}
  \end{enumerate}
  Then, the open lattices of $F$ are those of the form
  \begin{equation}
    \label{eqn:openlattices}
    \Lambda = \sum_{\alpha} \pp^{n(\alpha)} t^\alpha.
  \end{equation}
\end{prop}

\begin{rmk}
  Let us clarify what the second part in condition (ii) above stands
  for. The condition is that for any $l \in \left\{ 1,\ldots,r \right\}$
  and fixed indices $i_{l+1}, \ldots, i_{n-1}$, given $d \in \ZZ$ there is an
  integer $k_0$ such that for every $k \geq k_0$ and $\alpha \in
  I(k,i_{l+1}, \ldots, i_{n-1})$ we have $n(\alpha) \leq d$.
\end{rmk}
 
\begin{proof}
  We will prove the result by induction on $n$. For $n=2$, see
  \cite[Propositions 3.1 and 3.7]{camara-fa2dlfs}. Suppose $n > 2$. Then write $L =
  K\curlyb{t_1}\cdots\curlyb{t_r}\roundb{t_{r+1}}\cdots\roundb{t_{n-2}}$,
  with $r \in \left\{ 0, n-2 \right\}$. 
  By induction hypothesis, the
  higher topology on $L$ is locally convex and its open lattices are of the form
  \begin{equation}
    \label{eqn:latticesinductionhyp}
    M = \sum_{\beta \in J} \pp^{n(\beta)} t^\beta,
  \end{equation}
  with $(n(\beta))_{\beta \in J} \subset \ZZ \cup \left\{ -\infty
  \right\}$ a net satisfying the conditions in the
  statement of the proposition.

  Now we need to distinguish two cases. First, if $r \leq n-2$, we must
  apply the construction in which neighbourhoods of zero are of the form
  (\ref{eqn:basicnhoodequalchar}), as $F = L\roundb{t_{n-1}}$. So we let
  \begin{equation*}
    M_{i} = \sum_{\beta \in J} \pp^{n(\beta, i)} t^\beta,\quad i \in
    \ZZ,
  \end{equation*}
  with the property that there is an $i_0 \in \ZZ$ such that for all
  $i \geq i_0$, $M_i = L$. This last condition is equivalent to setting
  $n(\beta, i) = -\infty$ for all $\beta \in J$ and $i \geq i_0$. As the
  $M_i$ describe a basis of neighbourhoods of zero for the higher
  topology on $L$, the higher topology on $F$ admits a basis of
  neighbourhoods of zero formed by sets of the form
  \begin{equation*}
    \Lambda = \sum_{i \in \ZZ} M_i t_{n-1}^i.
  \end{equation*}
  By induction hypothesis, the $M_i$ are all $\OO$-lattices, which is
  enough to show that $\Lambda$ is an $\OO$-lattice. So, in this case,
  we let $\alpha = \left( \beta,i \right)$ 
  so that a basis of neighbourhoods of zero for the higher topology is described by sets
  $\Lambda = \sum_{\alpha \in I} \pp^{n(\alpha)} t^\alpha$. On top of the
  conditions which the indices $n(\beta,i)$ satisfy by the induction 
  hypothesis for $\beta \in J$, we must add the further condition that
  there is an integer $i_0$ such that for all $i \geq i_0$, $n(\alpha) =
  -\infty$ for all $\alpha \in I(i)$. This shows that our claim holds in
  this case.

  The second case is the one in which $r > n-2$ and we must apply the
  construction in which neighbourhoods of zero are given by sets of the
  form (\ref{eqn:basicnhoodmixedchar}), as $F = L\curlyb{t}$. So we set
  \begin{equation*}
    M_{i} = \sum_{\beta \in J} \pp^{n(\beta, i)} t^\beta,\quad i \in \ZZ,
  \end{equation*}
  subject to the properties:
  \begin{enumerate}
    \item There is an integer $c$ such that for every $i \in \ZZ$,
      $\pp_L^c \subset M_i$. By induction hypothesis, this means that
      $n(\beta,i) \leq c$ for every $\beta \in J$ and $i \in \ZZ$.
    \item $M_i \to L$ as $i \to \infty$. This is equivalent to
      $n(\beta, i) \to -\infty$ for $\beta \in J$ as $i \to \infty$.
  \end{enumerate}
  As $M_i$ describe a basis of neighbourhoods of zero of the topology of
  $L$, the sets of the form
  \begin{equation*}
    \Lambda = \sum_{i \in \ZZ} M_i t_{n-1}^i
  \end{equation*}
  describe a basis of neighbourhoods of zero for the higher topology on
  $F$. Since the $M_i$ are $\OO$-lattices, we get that $\Lambda$ is an
  $\OO$-lattice. Again, we let $\alpha = (\beta,i)$,
  so that the $\OO$-lattice $\Lambda$ may be described as $\Lambda =
  \sum_{\alpha \in I} \pp^{n(\alpha)} t^\alpha$. On top of the conditions
  satisfied by the $n(\alpha)$ which are inherited by induction, there are the
  two new conditions:
  \begin{enumerate}
    \item There is an integer $c$ such that $n(\alpha) \geq c$ for all
      $\alpha \in I$.
    \item $n(\alpha) \to -\infty$ for $\alpha \in I(i)$, as $i \to
      \infty$.
  \end{enumerate}
  This shows that the result also holds in this case.
\end{proof}

After showing that the higher topology on $K \into F$ is locally convex, it is
natural to describe it in terms of seminorms.

\begin{prop}
  \label{prop:seminorms}
  The higher topology on $F$ is the locally convex $K$-vector space
  topology defined by the seminorms of the form
  \begin{equation}
    \label{eqn:admissibleseminorm}
    \| \cdot \|: F \to \mathbb{R}, \quad \sum_\alpha x(\alpha) t^\alpha
    \mapsto \sup_\alpha |x(\alpha)| q^{n(\alpha)} 
  \end{equation}
  as $( n(\alpha) )_{\alpha \in I} \subset \ZZ \cup \left\{ -\infty
  \right\}$ varies over the nets
  described in the statement of Proposition
  \ref{prop:highertopislocconv}.
\end{prop}

\begin{proof}
  It is necessary to show that the gauge seminorm associated to the open
  lattice $\Lambda$ as in (\ref{eqn:openlattices}) is the one given by
  (\ref{eqn:admissibleseminorm}).

  The gauge seminorm defined by $\Lambda$ is by definition
  \begin{equation*}
    \|x\| = \inf_{x \in a\Lambda} |a|,\quad \text{for } x \in F.
  \end{equation*}
  Let $x = \sum_{\alpha} x(\alpha) t^\alpha$. We have that $x \in a\Lambda$
  if and only if $x(\alpha) \in
  a\pp^{n(\alpha)}$ for every $\alpha \in I$. That is, if and only if
  \begin{equation}
    \label{eqn:pfadmissseminorm}
    |x(\alpha)| q^{n(\alpha)} \leq |a| \quad \text{for all } \alpha \in I.
  \end{equation}
  The infimum of the values $|a|$ for which
  (\ref{eqn:pfadmissseminorm}) holds is the supremum of the
  values $|x(\alpha)|q^{n(\alpha)}$ as $\alpha \in I$.
\end{proof}

\begin{df}
  The seminorms on $F$ defined in the previous proposition will be
  referred to as {\it admissible seminorms}. 
\end{df}
  
  An admissible seminorm
  $\| \cdot \|$ is attached to a net $( n(\alpha) )_{\alpha \in I} \subset \ZZ \cup \left\{ -\infty \right\}$. If
  we have chosen notation not to reflect this fact it is in pursue of
  a lighter reading and understanding that the net $( n({\alpha)}
  )_{\alpha \in I}$, when needed, will be clear from the context.


\section{First properties}
\label{sec:firstproperties}

We summarise some properties which were known already for
higher topologies, or which are deduced immediately from the fact that these
topologies are locally convex. We also state some properties which do
not hold in general because they are known not to hold already for $n=2$.

The field $F$, equipped with a higher topology, is a locally convex
$K$-vector space, as shown in Proposition \ref{prop:highertopislocconv}. As such, it is a topological vector space. It is a
previously known fact that higher topologies are Hausdorff. In order to show
that this property holds in our setting it is enough to show that, given
$x \in F^\times$, there is an admissible seminorm $\| \cdot \|$ for which
$\|x\| \neq 0$. If the $\alpha$-coefficient of $x$ is
nonzero, any admissible seminorm for which $n(\alpha) > -\infty$ suffices.

Moreover, the reduction map $\OO_F \to \overline{F}$ is open when
$\OO_F$ is given the subspace topology and $\overline{F}$ a higher
topology compatible with the choice of coefficient field if $\car F \neq
\car \overline{F}$ \cite[Proposition 3.6.(v)]{camara-top-rat-pts-hlfs-arxiv}.

Multiplication $F \to F$ by a fixed nonzero element induces a
homeomorphism of $F$, but multiplication $\mu: F \times F \to F$ is not
continuous \cite[\textsection 1.3.2]{ihlf}; the immediate reason why this is the
case being that for any open lattice $\Lambda$ we have $\mu(\Lambda,\Lambda) = F$.

Higher topologies are not first-countable and therefore not metrizable
\cite[\textsection 1.3.2]{ihlf}.
Moreover, in general, $F$ is not bornological, barrelled, reflexive nor
nuclear and its rings of integers are not c-compact nor compactoid, the
first counterexample being the field $K\curlyb{t}$ \cite{camara-fa2dlfs}.

\begin{rmk}
  \label{rmk:powerseriesconverge}
  Power series in the
  system of parameters $t_1, \ldots, t_{n-1}$ are convergent in the
  higher topology. 
  If $x = \sum_{\alpha \in I} x(\alpha) t^\alpha \in F$, we define a net
  $(s(\alpha))_{\alpha \in I} \subset F$ by taking $s(\alpha) =
  \sum_{\alpha' \leq \alpha} x(\alpha') t^{\alpha'}$. If $\| \cdot
  \|$ is an admissible seminorm on $F$, then as $\alpha \in I$ grows,
  the value $\| x - s(\alpha) \|$ becomes arbitrarily small.
\end{rmk}

\section{Bounded $\OO$-submodules}
\label{sec:boundedsubmodules}

Let us study the bounded $\OO$-submodules of $F$. We start by describing
a basis for the Von-Neumann bornology on $K \into F$.

\begin{prop}
  \label{prop:boundedsubmodules}
  Let $( k(\alpha) )_{\alpha \in I} \subset \ZZ \cup \left\{
  \infty \right\}$ be a net subjected to the conditions:
  \begin{enumerate}
    \item For every $l \in \left\{ r+1, \ldots, n-1 \right\}$ and indices
      $i_{l+1}, \ldots, i_{n-1} \in \ZZ$ there is an index $j_0 \in \ZZ$
      such that for every $j < j_0$ we have $k(\alpha) = \infty$ for all
      $\alpha \in I(j, i_{l+1}, \ldots, i_{n-1})$.
    \item For every $l \in \left\{ 1, \ldots, r \right\}$ and
      $i_{l+1}, \ldots, i_{n-1} \in \ZZ$, there is an integer $d$ such
      that $k(\alpha) \geq d$ for all $\alpha \in I(i_{l+1}, \ldots,
      i_{n-1})$.
  \end{enumerate}
  The Von-Neumann bornology of $F$ admits as a basis the collection of
  $\OO$-submodules
  \begin{equation}
    \label{eqn:boundedsets}
    B = \sum_{\alpha \in I} \pp^{k(\alpha)} t^\alpha
  \end{equation}
  as $(k(\alpha))_{\alpha \in I}$ varies over the nets specified above.
\end{prop}

\begin{proof}
  First, let us show that the sets $B$ are bounded. As we will use
  induction on $n$, the case $n=2$ is thoroughly explained in
  \cite[Propositions 4.2 and 4.4]{camara-fa2dlfs}. 
  
  Let $\| \cdot \|$ be
  an admissible seminorm attached to the net $( n(\alpha)
  )_{\alpha\in I} \subset \ZZ \cup \left\{ -\infty \right\}$.

  Let $x = \sum_{\alpha} x(\alpha) t^\alpha$. We have
  \begin{equation*}
    \| x \| = \sup_\alpha |x(\alpha)| q^{n(\alpha)} \leq \sup_\alpha
    q^{n(\alpha) -
    k(\alpha)}
  \end{equation*}
  and therefore it is enough to prove that the set
  $\left\{ n(\alpha) - k(\alpha) \right\}_{\alpha \in I} \subset \ZZ \cup
  \left\{ - \infty \right\}$ is bounded above.

  We distinguish two cases. Suppose first that $r \leq n-2$. Then $F =
  L\roundb{t_{n-1}}$. On one hand, there is an index
  $j_0 \in \ZZ$ such that $k(\alpha) = \infty$ for every $\alpha \in
  I(j)$, $j < j_0$. On the other hand, there is an index $j_1 \in \ZZ$
  such that $n(\alpha)=-\infty$ for every $\alpha \in I(j)$, $j > j_1$. It
  is therefore enough to show that each of the finitely many sets
  \begin{equation*}
    N(j) = \left\{ n(\alpha) - k(\alpha);\; \alpha \in I(j) \right\},\quad
    j_0 \leq j \leq j_1
  \end{equation*}
  are bounded above. But for
  each $j \in \ZZ$, the net $( n(\alpha) )_{\alpha \in I(j)}$ defines
  an admissible seminorm and $\sum_{\beta\in J} \pp^{k(\beta,j)}
  t^{\beta} \subset L$ is a bounded $\OO$-submodule; this implies the
  boundedness of $N(j)$.

  The case in which $r = n-1$, and therefore $F = L\curlyb{t_{n-1}}$, is
  simpler: we have that all the $n(\alpha)$ are bounded above and all the
  $k(\alpha)$ are bounded below; therefore the differences $n(\alpha) -
  k(\alpha)$ are bounded above.

  Second, we have to show that any bounded subset of $F$ is contained in
  an $\OO$-submodule of $F$ of the form (\ref{eqn:boundedsets}). 
  
  The elements of any bounded subset $D \subset F$ cannot have
  $t_{n-1}$-expansions with arbitrarily large coefficients in $L$ in a fixed
  degree: otherwise, suppose that this is not the case and that for $j
  \in \ZZ$ the $j$-th coefficients in the $t_{n-1}$-expansions of the
  elements in $D$ may be arbitrarily large. By choosing any admissible seminorm with $n(\alpha) >
  -\infty$ for some $\alpha \in I(j)$ we easily obtain that $D$ is not
  bounded. 
  
  Hence, $D$ is contained in an $\OO$-submodule of the form $C = \sum_{i \in \ZZ} C_i t_{n-1}^{i}$ with $C_i \subset L$ bounded $\OO$-submodules. By
  induction hypothesis, let us write
  \begin{equation*}
    C_i = \sum_{\beta \in J} \pp^{k(\beta,i)}t^\beta,\quad i \in \ZZ,
  \end{equation*}
  with $k(\beta,i) \in \ZZ \cup \left\{ \infty \right\}$ satisfying the
  conditions exposed in the statement of the proposition.

  By letting $\alpha = (\beta, i) \in I$, we may write $C = \sum_{\alpha \in I}
  \pp^{k(\alpha)}t^\alpha$ and we only have to show that the indices
  $k(\alpha)$ might be taken to satisfy the conditions exposed in the
  proposition. Suppose that this is not the case, and let us consider
  separate cases again.
  
  First, if $r \leq n-2$ and $F = L\roundb{t_{n-1}}$, the indices
  $k(\alpha)$ may be taken to satisfy condition 
  (ii) by induction hypothesis. Condition (i) is also
  satisfied by induction hypothesis for every $l \in \left\{ r+1, \ldots,
  n-2 \right\}$. So we only have to show that if the $k(\alpha)$ may not
  be taken to satisfy condition (i) in the case $l = n-1$, then $D$
  cannot be bounded. 
  

  If the condition does not hold, then there is a decreasing sequence
  $(j_h)_{h \geq 0} \subset \ZZ_{<0}$, an index $\alpha_h \in I(j_h)$ and
  an element $\xi_h \in D$ such that its $\alpha_h$-coefficient, which we
  label $x(\alpha_h)$, is nonzero. Let
  \begin{equation*}
    n(\alpha) = \left\{ 
    \begin{array}{ll}
      -j_h + v(x(\alpha_h)),& \text{ if } \alpha = \alpha_h \in I(j_h),\\
      -\infty & \text{ otherwise}.
    \end{array}
    \right.
  \end{equation*}
  The net $( n(\alpha) )_{\alpha \in I}$ defines
  an admissible seminorm $\| \cdot \|$. Since $\| \xi_h \| \geq
  |x(\alpha_h)| q^{n(\alpha_h)} = q^{-j_h}$, $D$ cannot be bounded.
  

  Second, suppose that $r=n-1$ and $F = L\curlyb{t_{n-1}}$. By induction
  hypothesis, we know that condition (ii) holds for every $l \in \left\{
  1, \ldots, n-2 \right\}$, so we suppose that it does not hold in the
  case $l = n-1$. In such case, at least one of the following must happen:
  \begin{enumerate}[1.]
    \item There is a decreasing sequence $(j_h)_{h \geq 0} \subset
      \ZZ_{<0}$, an index $\alpha_h \in I(j_h)$ and $\xi_{h} \in D$
      such that, if $x(\alpha_h)$ denotes the $\alpha_{h}$-coefficient of
      $\xi_h$, we have $|x(\alpha_h)| \to \infty$ as $h \to \infty$.
    \item There is an increasing sequence $(j_h)_{h \geq 0} \subset
      \ZZ_{\geq 0}$, an index $\alpha_h \in I(j_h)$ and $\xi_{h} \in D$
      such that, if $x(\alpha_h)$ denotes the $\alpha_{h}$-coefficient of
      $\xi_h$, we have $|x(\alpha_h)| \to \infty$ as $h \to \infty$.
  \end{enumerate}
  Suppose that condition 1 holds. In this case, let
  \begin{equation*}
    n(\alpha) = \left\{ 
    \begin{array}{ll}
      0,&\text{ if } \alpha=\alpha_h \text{ for some } h \geq 0,\\
      -\infty & \text{ otherwise. }
    \end{array}
    \right.
  \end{equation*}
  The net $( n(\alpha) )_{\alpha \in I}$ defines an
  admissible seminorm $\| \cdot \|$. Now, for $h \geq 0$, we have
  $\| \xi_h \| \geq q^{-v(x(\alpha_h))}$ and hence $D$ cannot be bounded. 
  



  Finally, if condition 1 does not hold, then condition 2 must happen. In
  such case, let
  \begin{equation*}
    n(\alpha) = \left\{ 
    \begin{array}{ll}
      (v(x(\alpha_h)) - 1)/2,&\text{ if } \alpha=\alpha_h
      \text{ for some } h \geq 0
      \text{ and } v(x(\alpha_h)) \text{ odd.}\\
      v(x(\alpha_h))/2,&\text{ if } \alpha=\alpha_h \text{ for some } h
      \geq 0
      \text{ and } v(x(\alpha_h)) \text{ even.}\\
      -\infty&\text{ otherwise.}
    \end{array}
    \right.
  \end{equation*}
  The net $( n(\alpha) )_{\alpha \in I}$ defines an
  admissible seminorm. Moreover, we have $n(\alpha_h) - v(x(\alpha_h))
  \to \infty$ as $h \to \infty$. We have that $\| \xi_h \| \geq
  |x(\alpha_h)|q^{n(\alpha_h)} = q^{n(\alpha_h) - v(x(\alpha_h))}$
  and therefore $D$ cannot be bounded.
 %
 %
\end{proof}

\begin{df}
  We will say that an $\OO$-submodule of the form (\ref{eqn:boundedsets})
  is a {\it basic bounded submodule} of $F$.
\end{df}

The following result is necessary in order to compare compactoids and
c-compacts in the sequel.

\begin{prop}
  \label{prop:boundedsubmodsarecomplete}
  The submodules $B$ in Proposition \ref{prop:boundedsubmodules} are
  complete.
\end{prop}

\begin{proof}
  Let $B = \sum_{\alpha \in I} \pp^{k(\alpha)} t^\alpha$ with $(
  k(\alpha) )_{\alpha \in I} \subset \ZZ \cup \left\{ \infty
  \right\}$ satisfying conditions (i) and (ii) in the statement of
  Proposition \ref{prop:boundedsubmodules}.

  Let $H$ be a directed set and $(x(h))_{h \in H} \subset B$ a Cauchy net.
  Let us denote, for each $h \in H$, $x(h) = \sum_\alpha x(h)(\alpha)
  t^\alpha$ with $x(h)(\alpha) \in \pp^{k(\alpha)}$ for $\alpha \in I$.

  We have that, for a fixed $\alpha \in I$, $(x(h)(\alpha))_{h \in H}
  \subset \pp^{k(\alpha)}$ is a Cauchy net. As $\pp^{k(\alpha)}$ is
  complete, the net converges to an element $x(\alpha) \in
  \pp^{k(\alpha)}$.

  If the power series $\sum_{\alpha} x(\alpha) t^\alpha$ defines an
  element $x$ in $F$, then $x \in B$ and $x(h) \to x$. This is easy to
  check by induction on $n$ (the case $n = 2$ may be found in
  \cite[\textsection 5]{camara-fa2dlfs}).
\end{proof}

As we have explained, multiplication $\mu: F \times F \to F$ is not a
continuous map. However, we may shown that it is bounded.

\begin{prop}
  Multiplication $\mu: F \times F \to F$ is a bounded map.
\end{prop}

\begin{proof}
  The argument for the proof is by induction on $n$. The case $n = 2$ is
  dealt with in \cite[Proposition 4.8]{camara-fa2dlfs}, and the same
  argument applies when looking at $F =L\curlyb{t_{n-1}}$ or $L\roundb{t_{n-1}}$
\end{proof}

\section{Compactoid $\OO$-submodules}
\label{sec:compactoidsubmodules}

The result below outlines which basic bounded submodules of $F$ are
compactoid, and thus describes a basis for the bornology on $F$ generated
by compactoid $\OO$-submodules.

\begin{prop}
  \label{prop:compactoidsubmodules}
  Let $(k(\alpha) )_{\alpha \in I} \subset
  \ZZ \cup \left\{ \infty \right\}$ be a net satisfying the
  conditions:
  \begin{enumerate}
    \item For every $l \in \left\{ r+1, \ldots, n-1 \right\}$ and indices
      $i_{l+1}, \ldots, i_{n-1} \in \ZZ$ there is an index $j_0 \in \ZZ$
      such that for every $j < j_0$ we have $k(\alpha) = \infty$ for all
      $\alpha \in I(j, i_{l+1}, \ldots, i_{n-1})$.
    \item For every $l \in \left\{ 1, \ldots, r \right\}$ and
      $i_{l+1}, \ldots, i_{n-1} \in \ZZ$, there is an integer $d \in \ZZ$
      such that $k(\alpha) \geq d$ for all $\alpha \in I(i_{l+1}, \ldots,
      i_{n-1})$ and we have that $k(\alpha) \to \infty$ for $\alpha \in
      I(j,i_{l+1}, \ldots, i_{n-1})$, as $j \to -\infty$.
  \end{enumerate}
  Then the $\OO$-submodule of $F$ given by 
  \begin{equation}
    \label{eqn:basiccpctoidmodule}
    B = \sum_{\alpha} \pp^{k(\alpha)} t^\alpha 
  \end{equation}
  is compactoid. 
  
  The $\OO$-submodules of the form (\ref{eqn:basiccpctoidmodule}) are the
  only compactoid submodules amongst basic bounded submodules, and define
  a basis for the bornology on $F$ defined by compactoid submodules.
\end{prop}

In the proof to this proposition we shall need to consider the projection
maps to the coefficients of the $\alpha$-expansions of elements of $F$. For 
every $\alpha_0 \in I$, consider the continuous linear form:
\begin{equation*}
  \pi_{\alpha_0}: F \to K,\quad \sum_{\alpha} x(\alpha) t^\alpha \mapsto
  x(\alpha_0).
\end{equation*}

\begin{proof}
  The result holds for $n = 2$ as shown in \cite[\textsection 5]{camara-fa2dlfs}.
  
  First, let us show that the submodule $B$ as in
  (\ref{eqn:basiccpctoidmodule}) is compactoid.
  Let $\Lambda = \sum_\alpha \pp^{n(\alpha)} t^\alpha$ be an open lattice.
  We will show that there exist elements $x_1, \ldots, x_m \in F$ such
  that $B \subset \Lambda + \OO x_1+\cdots+\OO x_m$.

  Regardless of the value of $r \in \left\{ 0, \ldots, n-1 \right\}$,
  there are two indices $j_0, j_1 \in \ZZ$ such that 
  \begin{equation}
    k(\alpha) \geq n(\alpha), \quad \text{for all } \alpha \in I(j) \text{
    with } j < j_0 \text{ or } j > j_1.
    \label{eqn:compactoidsubmodules1}
  \end{equation}
  if $j_0 > j_1$ then $B \subset \Lambda$ and we are done. Henceforth, we
  assume $j_0 \leq j_1$.

  Let us examine the situation for $j_0 \leq j \leq j_1$. For a fixed
  such $j$, let $\alpha = (\beta, j)$ with $\beta \in J$. By induction
  hypothesis, the $\OO$-submodule
  \begin{equation*}
    \sum_{\beta\in J} \pp^{k(\beta,j)}t^\beta \subset L
  \end{equation*}
  is compactoid. Similarly, for a fixed such $j$, $\sum_\beta
  \pp^{n(\beta,j)}t^\beta$ is an open lattice in $L$. Therefore, there
  exist a finite number of elements $y_{j,1},\ldots,y_{j,m_j} \in L$ for
  which we have, for $j_0 \leq j \leq j_1$,
  \begin{equation*}
    \sum_{\beta} \pp^{k(\beta,j)} t^\beta \subset \sum_\beta
    \pp^{n(\beta,j)} t^\beta + \OO y_{j,1}+\cdots+\OO 
    y_{j,m_j}.
  \end{equation*}
  Now, this implies that
  \begin{equation*}
    \sum_{j=j_0}^{j_1} \left( \sum_{\beta \in J} \pp^{k(\beta,j)} t^\beta
    \right) t_{n-1}^j \subset \sum_{j=j_0}^{j_1} \left( \sum_{\beta \in
    J} \pp^{n(\beta,j)} t^\beta + \sum_{s=1}^{m_j} \OO y_{j,s} \right)t_{n-1}^j.
  \end{equation*}
  We rewrite this last equation as:
  \begin{equation}
    \sum_{\substack{\alpha \in I(j)\\ j_0\leq j \leq j_1}} \pp^{k(\alpha)} t^\alpha
    \subset \sum_{\substack{\alpha \in I(j)\\j_0 \leq j \leq j_1}}
    \pp^{n(\alpha)}
    t^\alpha + \left( \sum_{j=j_0}^{j_1} \sum_{s=1}^{m_j} \OO y_{j,s}t_{n-1}^j \right).
    \label{eqn:compactoidsubmodules2}
  \end{equation}
  The fact that
  \begin{equation*}
    B \subset \Lambda +\left( \sum_{j=j_0}^{j_1} \sum_{s=1}^{m_j}
    \OO y_{j,s}t_{n-1}^j \right)
  \end{equation*}
  follows from (\ref{eqn:compactoidsubmodules1}) and
  (\ref{eqn:compactoidsubmodules2}).

  Second, let us show how any compactoid $\OO$-submodule of $F$ is
  contained in one of the form (\ref{eqn:basiccpctoidmodule}). Since
  compactoid $\OO$-submodules are bounded, it is enough to show that any
  basic bounded submodule $C = \sum_{\alpha} \pp^{k(\alpha)} t^\alpha$ is
  compactoid if and only if the indices $k(\alpha)$ satisfy conditions (i)
  and (ii) in the statement of the Proposition. We proceed by induction,
  the result holds for $n=2$ as mentioned above.

  Now, suppose $C$ is compactoid. Then, for every $j \in \ZZ$, the
  $\OO$-submodule of $L$ given by
  \begin{equation*}
    C_j = \sum_{\beta} \pp^{k(\beta,j)} t^\beta
  \end{equation*}
  is compactoid.

  Next, we distinguish cases. Suppose that $r \leq n-2$, so that $F =
  L\roundb{t_{n-1}}$. In such case, by induction hypothesis, we only need
  to check that the indices $k(\alpha)$ satisfy condition (i) for $l =
  n-1$. But if this condition does not hold, then from the proof of
  Proposition \ref{prop:boundedsubmodules} we deduce that $C$ cannot be
  compactoid, as it is not bounded. So there is nothing more to say in this case.
  
%
%
%
  Finally, suppose that $r =n-1$, so that $F=L\curlyb{t_{n-1}}$. By
  hypothesis induction, the indices $k(\alpha)$ satisfy condition (ii) in
  the statement of the Proposition for $1 \leq l \leq  n-2$. If condition
  (ii) for $l = n-1$ does not hold, then there is a decreasing sequence
  $(j_h)_{h \geq 0} \subset \ZZ_{<0}$ and an index $\alpha_h \in
  I(j_h)$ for each $h \geq 0$ such that the sequence
  $(k(\alpha_h))_{h \geq 0}$ is
  bounded above. Let $M \in \ZZ$ be such that $k(\alpha_h) < M$ for
  every $h \geq 0$. Let
  \begin{equation*}
    n(\alpha) = \left\{ 
    \begin{array}{ll}
      M, & \text{ if } \alpha = \alpha_h \text{ for some } h \geq 0,\\
      -\infty & \text{ otherwise }.
    \end{array} 
    \right.
  \end{equation*}
  The net $(n(\alpha))_{\alpha \in I}$ defines an open lattice $\Lambda =
  \sum_\alpha \pp^{n(\alpha)} t^\alpha$. Now, suppose that $x_1, \ldots,
  x_m \in F$ satisfy that $C \subset \Lambda + \OO x_1 + \cdots + \OO
  x_m$. Let us write, for $1 \leq l \leq m$, $x_l = \sum_\alpha
  x_l(\alpha) t^\alpha$.

  Since $x_l(\alpha) \to 0$ for $\alpha \in I(j)$ as $j \to -\infty$, we
  have that there is an index $k \in \ZZ$ such that for every $j \leq
  k$, we have $v(x_l(\alpha)) > M$ for every $\alpha \in I(j)$ and
  $1 \leq l \leq m$. Fix an $h \geq 0$ such that $j_h \leq k$. Now, for
  such an $h$, we have
  \begin{equation*}
    \pi_{\alpha_h}(C) \subset \pi_{\alpha_h}(\Lambda + \OO x_1 + \cdots +
    \OO x_m),
  \end{equation*}
  which implies
  \begin{equation*}
    \pp^{k(\alpha_h)} \subset \pp^M + \pp^{v(x_1(\alpha_h))} + \cdots +
    \pp^{v(x_m(\alpha_h))} = \pp^M.
  \end{equation*}
  This last inclusion implies that $M \leq k(\alpha_h)$, a
  contradiction. Hence, we must have $k(\alpha) \to \infty$ for $\alpha
  \in I(j)$ as $j \to \infty$.
\end{proof}

\begin{df}
  We will refer to the $\OO$-submodules of $F$ of the form
  (\ref{eqn:basiccpctoidmodule}) as {\it basic compactoid submodules of
  $F$}.
\end{df}

\begin{corollary}
  The basic compactoid $\OO$-submodules of $F$ are c-compact.
\end{corollary}

\begin{proof}
  An $\OO$-submodule of a locally convex $K$-vector space is c-compact
  and bounded if and only if it is compactoid and complete
  \cite[Prop. 12.7]{schneider-non-archimedean-functional-analysis}. So
  the result follows from the fact that these $\OO$-submodules are
  bounded, compactoid and complete.
\end{proof}

\section{Duality}
\label{sec:duality}

Let us discuss some issues regarding the dual space of $F$. We showed in
\cite[Theorem 6.2]{camara-fa2dlfs} how in the two-dimensional case $F$ is isomorphic in the category
of locally convex $K$-vector spaces to $F'_c$, its continuous dual space
topologized using the c-topology, that is: the topology of uniform
convergence on compactoid submodules. 

The c-topology is defined on the continuous dual of any n-dimensional $F$ by the
collection of seminorms
\begin{equation*}
  |\cdot|_B : F' \to \mathbb{R},\quad l \mapsto \sup_{x \in B} |l(x)|,
\end{equation*}
for any basic compactoid submodule $B \subset F$.

Our first goal in this section is to construct an isomorphism of locally
convex vector spaces $F \cong F'_c$.

We have already come across some continuous nonzero linear forms on $F$
in the previous section, we recall that these were the projections, for
every $\alpha_0 \in I$:
\begin{equation*}
  \pi_{\alpha_0} : F \to K, \quad \sum_{\alpha \in I} x(\alpha) t^\alpha
  \mapsto x(\alpha_0).
\end{equation*}

In particular, denote by $\pi_0$ the continuous linear form on $F$
constructed as in the previous example for $\alpha_0 = \left( 0, \ldots,0
\right) \in I$.

We relate $F$ and its continuous dual space. Define
\begin{equation}
  \gamma: F \to F', \quad x \mapsto \pi_x,
  \label{eqn:definegamma}
\end{equation}
with
\begin{equation*}
  \pi_x: F \to K,\quad y \mapsto \pi_0(xy).
\end{equation*}

\begin{lem}
  If $x = \sum_{\alpha \in I} x(\alpha) t^\alpha$ and $y = \sum_{\alpha
  \in I} y(\alpha) t^\alpha$ are elements in $F$, we have that
  \begin{equation*}
    \pi_x(y) = \sum_{\alpha \in I} x(-\alpha)y(\alpha) = \sum_{i_{n-1} \in
    \ZZ}\cdots\sum_{i_1 \in \ZZ} x_{(-i_1,\ldots, -i_{n-1})}
    y_{(i_1, \ldots, i_{n-1})}.
  \end{equation*}
\end{lem}

\begin{proof}
  The result becomes clear once notation is unwinded and the 0-th
  coefficient of $xy$ is taken for each parameter $t_l$ separately, for
  $l$ descending from $n-1$ to $1$.
\end{proof}

\begin{lem}
  \label{lem:surjectivityofgamma}
  Let $w \in F'$. Define, for each $\alpha \in I$, $x(\alpha) =
  w(t^{-\alpha}) \in K$. Then, the expression $\sum_{\alpha \in I}
  x(\alpha) t^\alpha$ defines an element in $F$.
\end{lem}

\begin{proof}
  The result may be shown by induction; the argument is the same as in
  \cite[Lemma 6.3]{camara-fa2dlfs}.
\end{proof}

\begin{thm}
  \label{thm:selfduality}
  The map $\gamma: F \to F'_c$ is an isomorphism of locally convex vector
  spaces.
\end{thm}

\begin{proof}
  The map $\gamma$ is linear and injective. Surjectivity follows from
  Lemma \ref{lem:surjectivityofgamma}; as if $w \in F'$, we apply the
  lemma to obtain $x = \sum_{\alpha \in I} x(\alpha) t^\alpha \in F$.
  Then, for any $y = \sum_{\alpha\in I} y(\alpha) t^\alpha \in F$, we have
  \begin{equation*}
    w(y) = w \left( \sum_{\alpha} y(\alpha) t^\alpha \right) =
    \sum_\alpha y(\alpha) w(t^\alpha) = \sum_\alpha y(\alpha) x(-\alpha) =
    \pi_x(y).
  \end{equation*}
  The second equality follows from Remark \ref{rmk:powerseriesconverge}.

  In order to show bicontinuity of $\gamma$, the argument is very similar
  to the one given in the proof of
  \cite[Theorem 6.2]{camara-fa2dlfs}; given a basic compactoid
  $\OO$-submodule $B = \sum_{\alpha \in I} \pp^{k(\alpha)} t^\alpha$ of
  $F$ as in Proposition \ref{prop:compactoidsubmodules}, the net
  $( -k(-\alpha) )_{\alpha \in I}$ defines an
  admissible seminorm $\| \cdot \|$ on $F$. We have that, for $x \in F$,
  \begin{equation*}
    \| x \| \leq q^{n} \text{ if and only if } |\pi_x|_B \leq q^n,
  \end{equation*}
  which concludes the proof.
\end{proof}

If $A \subset F$ is an $\OO$-submodule, we denote $A^\gamma =
\gamma^{-1}(A^p) \subset F$, with $A^p \subset F'$ being the pseudo-polar of $A$.

\begin{prop}
  \label{prop:pseudopolars}
  Let $A = \sum_{\alpha \in I} \pp^{k(\alpha)} t^\alpha \subset F$ be an
  $\OO$-submodule, with $k(\alpha) \in \ZZ \cup \left\{ \pm \infty
  \right\}$. We have that
  \begin{equation*}
    A^\gamma = \sum_{\alpha \in I} \pp^{1-k(-\alpha)} t^{\alpha}.
  \end{equation*}
\end{prop}

\begin{proof}
  The argument is the same as the one exposed in the proof of
  \cite[Proposition 6.7]{camara-fa2dlfs}.
\end{proof}

The isomorphism $\gamma$ is not unique as, for example, choosing
$\pi_\alpha$ for any $\alpha \in I$ instead of $\pi_0$ in the definition
of $\gamma$ would have given a different isomorphism. Thus, the shape of
$A^\gamma$ depends ultimately on our choice of $\gamma$.

However, there are certain facts which are general for pseudo-polars of
$\OO$-submodules in any locally convex $K$-vector space. As such, we
recall that taking the pseudo-polar exchanges open lattices and compactoid
$\OO$-submodules, and that the pseudo-bipolar of an $\OO$-submodule is equal
to its closure.

These facts are highlighted in the previous Proposition for the
$\OO$-submodules of the form $\sum_\alpha \pp^{k(\alpha)}t^\alpha$, which
are closed. The facts that for an open lattice $\Lambda$ we have that
$\Lambda^\gamma$ is compactoid and that for a basic compactoid set
$B$ we have that $B^\gamma$ is an open lattice are evident by checking
that for the nets $(n(\alpha))_{\alpha \in I}$ and
$(1-n(-\alpha))_{\alpha \in I}$, one of them satisfies conditions (i) and
(ii) in Proposition \ref{prop:highertopislocconv} if and only if the
other one satisfies conditions (i) and (ii) in Proposition
\ref{prop:compactoidsubmodules}.

%

\section{The general case}
\label{sec:generalcase}

In this section, let $K \into F$ be a general $n$-dimensional local field
over $K$. By structure
theory, there is an $r \in \left\{ 0, \ldots, n \right\}$ such that $F$
is a finite extension of
$F_0 :=
K\curlyb{t_1}\cdots\curlyb{t_r}\roundb{t_{r+1}}\cdots\roundb{t_{n-1}}$ as
explained in \textsection \ref{sec:categoryofhdlfs}.
Denote the degree of such extension by $e$.
The higher topology on $F$ may be defined as the product topology on $F
\cong
(F_0)^{e}$
\cite[Remark after 1.3.2]{ihlf}.

Since the product topology on a product of locally convex vector spaces
is again locally convex, a higher topology on $F$ is locally convex.
Open lattices (resp. continuous seminorms) on $F$ may be described using
Proposition \ref{prop:highertopislocconv} (resp. Proposition
\ref{prop:seminorms}) and \cite[Proposition 1.1]{camara-fa2dlfs}.

Finally, from Theorem \ref{thm:selfduality} we recover the existence of
an isomorphism $F \cong F'_c$ via the chain
\begin{equation*}
  F' \cong (F_0^e)' \cong (F'_0)^e \cong ((F_0)'_c)^e \cong F_0^e \cong
  F.
\end{equation*}
Explicit nonzero continuous linear forms on
$F$ may be obtained by composing the forms $\pi_\alpha: F_0 \to K$ for
$\alpha \in I$ with $\Tr_{F\vert F_0}$.

\section{Other types of higher local fields}

Let us center our attention, for the sake of completeness, on the higher
local fields which we have not treated in the previous.

First, suppose that $\car F = p$. In such case, as explained in
\textsection \ref{sec:categoryofhdlfs}, there is a finite field
$\mathbb{F}_q$ and elements $t_1, \ldots, t_n \in F$ such that
\begin{equation*}
  F \cong \mathbb{F}_q\roundb{t_1}\cdots\roundb{t_n}.
\end{equation*}
The field $\mathbb{F}_q\roundb{t_1}$ is a local field and the results in
this work may be applied to $F$ if we let $K =
\mathbb{F}_q\roundb{t_1}$. However, after \cite[\textsection 9]{camara-fa2dlfs}, we are only stating that the higher topology
on $F$ is a linear topology when we regard $F$ as a vector space over
$\mathbb{F}_q$. The choice between linear-topological structures over
$\mathbb{F}_q$ or locally convex structures over
$\mathbb{F}_q\roundb{t_1}$ is merely a matter of language in our case.

Now let $\mathbb{K} = \mathbb{R}$ or $\mathbb{C}$ and denote the usual
absolute value by $| \cdot |$. The theory of higher local fields is also
developed by looking at complete discrete valuation fields $F$ that have
an $n$-dimensional structure on them and such that $F_1 =
\mathbb{K}$. For these, there are $t_1, \ldots, t_{n-1} \in F$ for which
\begin{equation*}
  F \cong \mathbb{K}\roundb{t_1}\cdots\roundb{t_{n-1}}.
\end{equation*}
As hinted at in \cite[\textsection 8]{camara-fa2dlfs}, we can apply
the archimedean theory of locally convex spaces to study these fields.

The open disks
\begin{equation*}
  D_\rho = \left\{ a \in \mathbb{K};\; |a| < \rho \right\},\quad \rho \in
  \mathbb{Q}_{>0} \cup \left\{ \infty \right\}
\end{equation*}
supply a basis of convex sets for the euclidean topology on
$\mathbb{K}$.

The higher topology on $F$ is constructed by iterating the construction
in \cite[\textsection 8]{camara-fa2dlfs}.

\begin{prop}
  Let $I = \ZZ^{n-1}$ and $( \rho(\alpha) )_{\alpha \in I}
  \subset \mathbb{Q}_{>0} \cup \left\{ \infty \right\}$ be a net restricted to
  the condition:

  For any $l \in \left\{ 1, \ldots, n-1 \right\}$ and fixed indices
  $i_{l+1}, \ldots, i_{n-1} \in \ZZ$ there is a $k_0 \in \ZZ$ such that
  for every $k \geq k_0$ we have $\rho(\alpha) = \infty$ for all
  $\alpha \in I(k, i_{l+1}, \ldots, i_{n-1})$.

  The higher topology on $F$ is locally convex and it is defined by the
  seminorms of the form
  \begin{equation}
    \| \cdot \|: F \to \mathbb{R}, \sum_{\alpha \in I} x(\alpha) t^\alpha
    \mapsto \sup_{\alpha \in I} \frac{|x(\alpha)|}{\rho(\alpha)},
    \label{eqn:archiseminorm}
  \end{equation}
  with the convention that $a/\infty = 0$ for any $a \in
  \mathbb{R}_{\geq 0}$.
\end{prop}

\begin{proof}
  The result follows from \cite[Proposition 8.3]{camara-fa2dlfs}
  and straightforward adaptation of the arguments used in Proposition
  \ref{prop:highertopislocconv} and Corollary \ref{prop:seminorms}.
\end{proof}

\label{sec:positivecharcase}
\label{sec:archimedeancase}

\section{Future work}
\label{sec:future}

Let us start this discussion by saying that
\cite[\textsection 10]{camara-fa2dlfs} contains a list of topics
worth studying, and that many of them are closely related to the topics
dealt with in this note.

Among the directions outlined there, there is one which particularly has
a direct impact on the study of functional analytic properties of higher local
fields of arbitrary dimension. 

Structure and topology on higher local fields may be studied successfully
as an iteration of applications of inverse limits
in the form of completions and direct limits in the form of
localizations. In order to describe functional analytic structures that
hold in any dimension and regardless of the characteristic type of $F$,
it seems that two initial ingredients are necessary: a theory of locally
convex $\OO$-modules and a study of which functional analytic properties of
these modules are preserved after taking direct and inverse limits.

On a different direction, the dependence of higher topologies on choices
of coefficient fields as soon as $\car \overline{F} = 0$ is a well-known
handicap of the theory. For this reason, showing a class of subsets of
$F$ with an interesting topological or functional analytic property and
which would remain stable under change of coefficient field would be
extremely important.

Similarly, we have not dealt with the different rings of integers
of a higher local field in this note, on purpose: although they are
$\OO$-submodules which are very relevant for arithmetic purposes, as 
highlighted already by comparing $K\roundb{t}$ and $K\curlyb{t}$ in
\cite{camara-fa2dlfs}, the functional analytic properties of such rings
change drastically according to the characteristic of the residue field.
It would also be interesting to establish whether there is a relevant 
topological or analytic property which highlights the relevance of these
arithmetically interesting $\OO$-submodules.

\def\cprime{$'$} \def\cprime{$'$}


\vskip 1cm
\begin{flushright}
	{\bf Alberto C\'amara}\\
	\footnotesize{
	{\it School of Mathematical Sciences\\
	University of Nottingham\\
	University Park\\
	Nottingham\\
	NG7 2RD\\
	United Kingdom\\}}
	\url{http://www.maths.nottingham.ac.uk/personal/pmxac}\\
	\href{mailto:pmxac@nottingham.ac.uk}{{\tt pmxac@nottingham.ac.uk}}
\end{flushright}

\end{document}